\documentclass{amsart}

\usepackage{color,graphicx,amssymb,latexsym,amsfonts,txfonts,amsmath,amsthm}
\usepackage{pdfsync}
\usepackage{amsmath,amscd}
\usepackage[all,cmtip]{xy}

\usepackage{hyperref}
\hypersetup{
    colorlinks=true,       
    linkcolor=blue,          
    citecolor=blue,        
    filecolor=blue,      
    urlcolor=blue           
}

\input epsf
\newcommand{\s}{\vspace{0.1cm}}

\newtheorem{theo}{Theorem}
\newtheorem{coro}{Corollary}
\newtheorem{prop}{Proposition}
\newtheorem{lemm}{Lemma}

\theoremstyle{remark}
\newtheorem{rema}{\bf Remark}

\begin{document}

\title{Geometric description of Virtual Schottky groups}
\author{Ruben A. Hidalgo}


\address{Departamento de Matem\'atica y Estad\'{\i}stica, Universidad de La Frontera. Temuco, Chile}
\email{ruben.hidalgo@ufrontera.cl}
\thanks{Partially supported by Project Fondecyt 1190001}

\begin{abstract}
A virtual Schottky group is a Kleinian group $K$ containing a Schottky group $G$ as a finite index normal subgroup.
These groups correspond to those groups of automorphisms of closed Riemann surfaces which can be realized at the level of their lowest uniformizations.  In this paper we provide a geometrical structural decomposition of $K$. When $K/G$ is an abelian group, an explicit free product decomposition in terms of Klein-Maskit's combination theorems is provided.
\end{abstract}

\maketitle

\section{Introduction} 
Let $S$ be a closed Riemann surface of genus $g \geq 2$ and let ${\rm Aut}(S)$ (respectively, ${\rm Aut}^{+}(S)$) be its group of all conformal and anticonformal (respectively, conformal) automorphisms. In general ${\rm Aut}^{+}(S)={\rm Aut}(S)$ (for the generic situation this is the trivial group); otherwise, ${\rm Aut}^{+}(S)$ has index two in ${\rm Aut}(S)$. In \cite{Schwarz}, Schwarz  observed that ${\rm Aut}^{+}(S)$ is finite and, in \cite{Hurwitz}, Hurwitz noted that $|{\rm Aut}^{+}(S)| \leq 84(g-1)$. 

The surface $S$ can be uniformized by Kleinian groups. More precisely, there are Kleinian groups $K$ admitting an invariant connected component $\Delta$ of its region of discontinuity ($K$ acting freely on $\Delta$) such that $S$ is biholomorphic to the quotient Riemann surface $\Delta/K$. In this setting, one may wonder for the realization of subgroups of ${\rm Aut}(S)$ by them. These uniformizations of $S$ are partially ordered, the highest ones given by the Fuchsian uniformizations and the lowest ones being the Schottky uniformizations. 

Klein-Koebe-Poincar\'e's (Fuchsian) uniformization theorem \cite{Koebe2, Koebe3, Poincare} asserts the existence of a co-compact Fuchsian group $\Gamma \cong \pi_{1}(S)$, acting on the hyperbolic plane ${\mathbb H}^{2}$, and of a holomorphic regular covering map $Q:{\mathbb H}^{2} \to S$ whose deck group is $\Gamma$. In this case, ${\rm Aut}(S)$ lifts under $Q$ to obtain a NEC (non-Euclidean crystallographic)  group $N<{\rm Aut}({\mathbb H}^{2})$, containing $\Gamma$ as a normal subgroup, such that  ${\rm Aut}(S) \cong N/\Gamma$,  ${\rm Aut}^{+}(S) \cong N^{+}/\Gamma$, where $N^{+}$ is the subgroup of $N$ consisting of its conformal elements. The structure of NEC groups is well known \cite{Wilkie}.

Koebe's retrosection theorem \cite{Koebe1} asserts the existence of a Schottky group $G$ (a purely loxodromic Kleinian group isomorphic to a free group, of rank $g$, with non-empty region of discontinuity $\Omega$) and of a holomorphic covering map $P:\Omega \to S$ whose deck group is $G$. As $\Omega$ is not 
simply-connected, it may happen that some automorphisms of $S$ do not lift under $P$ to automorphisms of $\Omega$. In fact, if $H<{\rm Aut}^{+}(S)$ does lifts, then 
$|H|\leq 12(g-1)$ (for a $3$-dimensional combinatorial argument see \cite {Zimmermann2, Zimmermann3} and, for a Kleinian groups argument, see \cite{Hidalgo:cota}). This, in particular, asserts that if $H$ lifts with respect to $P$, then $S/H$ cannot be an orbifold of genus zero with exactly three cone points (this fact can also be obtained as a consequence of the results in \cite{Kra}). 
Now, if $H$ lifts, then we obtain an (extended) Kleinian group $K$, containing the Schottky group $G$ as a normal subgroup, and with $H \cong K/G$; we say that $K$ is an {\it (extended) virtual Schottky group}. (These groups are the correspondent to NEC groups appearing at the level of the Fuchsian uniformizations.)

A simple geometrical necessary and sufficient condition for the lifting of $H<{\rm Aut}(S)$ under a regular covering $P:\Omega \to S$, whose deck group is a Schottky group, is given by the equivariant loop theorem, i.e., the existence of a certain $H$-invariant collection ${\mathcal F} \subset S$ of pairwise disjoint simple loops such that $S \setminus {\mathcal F}$ consists of planar regions (see Theorem \ref{teo:lift}). This fact was proved in \cite{Y-M1, Y-M2} by using the theory of minimal surfaces. In \cite{H-M} there is a simple argument in the setting of Kleinian groups. 

Maskit's structural description of (extended) function groups \cite{Maskit:function1, Maskit:function, H:extendedfunction} permits to obtain a general structural description, in terms of Klein-Maskit's combination theorems \cite{Maskit:Comb, Maskit:Comb4}, of an (extended) virtual Schottky group $K$  (see Propertiess \ref{corovirtual} and \ref{principal}).  If $G$ is a Schottky group, which is a finite index normal subgroup of $K$, and we know the algebraic structure of $H=K/G$, then we should expect such an structural decomposition to be more explicit in a geometrical sense. In \cite{Hidalgo:Schottky}, this was done for $H<{\rm Aut}^{+}(S)$ a cyclic group (we recall it in Section \ref{Sec:casociclico}). In this paper, we provide such an explicit description when $H<{\rm Aut}^{+}(S)$ is an abelian group (Theorem \ref{Abel}).  Its proof, given in Section \ref{Sec:prueba}, is summarized as follows. We start with a virtual Schottky group $K$ and a Schottky group $G$, being normal subgroup of $K$ and such that $H=K/G$ is an abelian group. We fix a regular holomorphic covering $P:\Omega \to S=\Omega/G$. By Theorem \ref{teo:lift}, there is a family of loops ${\mathcal F} \subset S$ such that (i) $S \setminus {\mathcal F}$ is a collection of planar surfaces, (ii) ${\mathcal F}$ is $H$-invariant and (iii) ${\mathcal F}$ defines the covering $P$. We lift, under $P$, such a collection of loops to obtain a collection of loops $\widehat{\mathcal F} \subset \Omega$ (called structural loops and the components of $\Omega \setminus \widehat{\mathcal F}$ called structural regions). As ${\mathcal F}$ is $H$-invariant, it follows that $\widehat{\mathcal F}$ is $K$-invariant. Each structural region has $K$-stabilizer being either trivial, finite cyclic or isomorphic to ${\mathbb Z}_{2}^{2}$. We observe that each of these stabilizers can be enlarged to one of the so called basic virtual Schottky groups (these are described in Section \ref{Sec:basicos}), which happens to be either the same groups or some HNN-extension by some loxodromic elements. Then, we proceed 
to glue a finite number of these structure regions, along some of these structural loops (with trivial $K$-stabilzers), in order to obtain a larger connected set $\widetilde{R}$ (with a finite number of structural boundary loops).  We apply a free product, in the sense of Klein-Maskit's combination theorem, of the involved basic virtual Schottky groups to obtain a Kleinian group $K^{*}<K$. Next, we consider the boundary structural loops of $\widetilde{R}$, with trivial $K$-stabilizers, and we observe that they are paired by some loxodromic elements of $K$. We produce some HNN-extensions of $K^{*}$, in the sense of  Klein-Maskit's combination theorem, by these loxodromic elements to obtain $K$.

From a $3$-dimensional point of view, Schottky groups are  exactly those producing a geometrically finite complete hyperbolic structures on the interior of a handledody with injectivity radius bounded away from zero. In this way, a geometrical structure description of (extended) virtual Schottky groups also provides a description of finite group actions on handlebodies \cite{Zimmermann2, Zimmermann3} from a point of view of Kleinian groups.

\section{Preliminaries}\label{Sec:preliminares}
In this section we recall some generalities on (extended) Kleinian groups, we state Klein-Maskit's combination theorem and  Maskits' structural description of (extended) function groups. The basic details may be found, for instance, in the books \cite{Maskit:book,MT}. 

\subsection{(Extended) Kleinian groups}
 The group of conformal automorphisms of the Riemann sphere $\widehat{\mathbb C}$ is the group of M\"obius transformations ${\mathbb M} \cong {\rm PSL}_{2}({\mathbb C})$ and its group of conformal and anticonformal automorphisms 
is $\widehat{\mathbb M}=\langle {\mathbb M}, J(z)=\overline{z}\rangle$ (the  elements of $\widehat{\mathbb M}\setminus{\mathbb M}$ are called {\it extended M\"obius transformations}). If $K \leq \widehat{M}$, then we set $K^{+}:=K \cap {\mathbb M}$.    

A discrete subgroup $K$ of ${\mathbb M}$ (respectively, of $\widehat{\mathbb M}$ and containing extended M\"obius transformations) is called a Kleinian group (respectively, an extended Kleinian group).  In this case, the {\it region of discontinuity} of $K$ is the (which it might be empty) open set $\Omega \subset \widehat{\mathbb C}$ consisting of all those points $p \in \widehat{\mathbb C}$ on which $K$ acts discontinuously (i.e., (i) its $K$-stabilizer $K(p)$ is finite and (ii) there is an open set $U$, containing $p$, such that $A(U) \cap U = \emptyset$ if $T \in K \setminus K(p)$). The complement $\Lambda:=\widehat{\mathbb C} \setminus \Omega$ is called its {\it limit set}.

\subsection{Schottky groups}
The Schottky group of rank zero is just the trivial group.
A {\it Schottky group of rank $g \geq 1$} is a group $G$ generated by $g$ loxodromic elements $A_{1},\ldots, A_{g}$ such that: (i) there exists a collection of
$2g$ pairwise disjoint simple loops $C_{1},\dots,C_{g}$, $C'_{1}, \ldots,C'_{g}$ on the the Riemann sphere $\widehat{\mathbb C}$,
bounding a common region $\mathcal D$ of connectivity $2g$,  (ii) $A_{j}(C_{j})=C'_{j}$ and (iii) $A_{j}({\mathcal D}) \cap {\mathcal D} = \emptyset$, for all $j=1,\ldots,g$.  
The set of transformations $A_{1},\ldots,A_{g}$ as above is called a {\it Schottky set of generators} for $G$.

It is known that a Schottky group of rank $g$ is a purely loxodromic Kleinian group, isomorphic to a free group of rank $g$, without empty region of discontinuity (the converse holds \cite{Maskit2}). In \cite{Chuckrow},  Chuckrow proved that every set of $g$ generators of a Schottky group $G$ of rank $g$ is in fact a Schottky set of generators. Its region of discontinuity $\Omega$ is connected and the quotient space $\Omega/G$ is a closed Riemann surface of genus $g$. 

Koebe's retrosection theorem \cite{Koebe1} states that every closed Riemann surface is biholomorphically equivalent to $\Omega/G$ for a suitable Schottky group $G$. A simple proof of this fact, using quasiconformal mappings theory, was provided by Bers \cite{Bers}.

\subsection{(Extended) virtual Schottky groups}
An (extended) Kleinian group is called an {\it (extended) virtual Schottky group} if it contains a Schottky group as a finite index subgroup. The finite index condition permits to assume the Schottky group to be a finite index normal subgroup. In the other direction, if an (extended) Kleinian group contains a Schottky group of positive rank as a normal subgroup, then it must have finite index.

\subsection{Schottky uniformizations of closed Riemann surfaces}
A {\it Schottky uniformization} of a closed Riemann surface $S$ is a triple $(\Omega,G,P)$, where $G$ is a Schottky group (necessarily of rank equal to the genus of $S$) with region of discontinuity $\Omega$, and $P:\Omega \to S$ is a regular covering with $G$ as its deck group.

\begin{theo}[\cite{Maskit1}]\label{Thm:Uniform}
Let $S$ be a closed Riemann surface of genus $g$.

\noindent
{\rm (1)} If $(\Omega,G,P)$ is a Schottky uniformization of $S$, then there exists a collection $\{\alpha_{m}\}$ of (homotopically independent) pairwise disjoint simple loops on $S$, with $S\setminus \{\alpha_{m}\}$ a collection of planar surfaces, such that 
$P:\Omega \to S$ is a regular covering for which:
{\rm (1.1)} each of the loops $\alpha_{m}$ lifts to loops, and
{\rm (1.2)} every loop in $\Omega$ is freely homotopic to a product of such lifted loops.

\noindent
{\rm (2)} Given any collection $\{\alpha_{m}\}$ of (homotopically independent) pairwise disjoint simple loops on $S$, with $S\setminus \{\alpha_{m}\}$ a collection of planar surfaces, there is a Schottky uniformization $(\Omega,G,P)$ of $S$ such that $P:\Delta \to S$ satisfies {\rm (1.1)} and {\rm (1.2)} above.
\end{theo}

The collection of loops $\{\alpha_{m}\}$, as in the above theorem, is called a {\it defining set of loops} for the Schottky uniformization $(\Omega,G,P)$.

\subsection{Lifting automorphisms to Schottky uniformizations}
Let $(\Omega,G,P)$ be a Schottky uniformization of a closed Riemann surface $S$.  
We say that a group $H<{\rm Aut}(S)$ {\it lifts with respect to} $(\Omega,G,P)$ if, for every $h \in H$, there is some $\widehat{h} \in {\rm Aut}(\Omega)$ such that $P \circ \widehat{h}=h \circ P$. 

In such a case, by lifting all the elements of $H$ provides of a discrete group $K<{\rm Aut}(\Omega)$ containing $G$ as a normal subgroup such that $H=K/G$, and a short exact sequence $1 \to G \to K \to H \to 1$. 
As the region of discontinuity of a Schottky group is of class $O_{AD}$; that is, it admits no holomorphic function with finite Dirichlet norm (see \cite[pg 241]{A-S}), it follows from this (see \cite[pg 200]{A-S}) that every conformal map from $\Omega$ into the Riemann sphere is a M\"obius transformation. This in particular asserts that the conformal (respectively, anticonformal) automorphisms of $\Omega$ are restrictions of M\"obius (respectively, extended M\"obius) transformations. In this way, $K$ is an (extended) virtual Schottky group. 

\subsection{Equivariant loop theorem}
Necessary and sufficient conditions, for a subgroup $H<{\rm Aut}(S)$ to lift with respect to a Schottky uniformization of $S$, is provided by Meeks-Yau's equivariant loop theorem \cite{Y-M1,Y-M2}, whose proof is based on minimal surfaces theory. In \cite{H-M} there is provided a proof which only uses techniques of Kleinian groups (in the same paper, a general equivariant loop theorem was stated for Kleinian groups).

\begin{theo}[Equivariant loop theorem for handlebodies]\label{teo:lift}
Let $(\Omega,G,P)$ be a Schottky uniformization of a closed Riemann surface $S$ and let $H<{\rm Aut}(S)$ be a (finite) group.
Then $H$ lifts with respect to $(\Omega,G,P)$ if and only if there is a collection of defining  loops ${\mathcal F}=\{\alpha_{m}\}$ of the uniformization, which is invariant under $H$, that is, for every $h \in H$ and every $m$, there exists $m'$ with $h(\alpha_{m})=\alpha_{m'}$. 
\end{theo}

A collection of loops ${\mathcal F}$, as in the previous theorem, is called a {\it Schottky system of loops of $H$ corresponding the Schottky uniformization $(\Omega,G,P)$}.

 We should note that such a Schottky system of loops of $H$ needs not to be unique, even if we require it to be {\it minimal} (that is, no non-trivial sub-collection still a Schottky system of loops of $H$).

\begin{rema}[A remark on the decomposition structure of $H$]\label{estructuradeH}
Let ${\mathcal F} \subset S$ be a collection of loops which is invariant under $H$ and $S\setminus {\mathcal F}$ consists of planar surfaces (i.e., a Schottky system of loops for $H$). Such a collection permits to describe an algebraic decomposition structure of $H$, as a finite iteration of amalgamated free products and HNN-extensions of certain subgroups of $H$, as follows. Let us consider a maximal collection of components of $S \setminus {\mathcal F}$, say $S_{1}$\ldots, $S_{n}$, so that any two different components are not $H$-equivalent. Let us denote by $H_{j}$ the $H$-stabilizer of $S_{j}$. It is possible to chose these surfaces so that, by adding some common boundary loops, we obtain a planar surface $S^{*}$ (containing each $S_{j}$ in its interior). If two surfaces $S_{i}$ and $S_{j}$ have a common boundary in $S^{*}$, then $H_{i} \cap H_{j}$ is either trivial or a cyclic group (this being exactly the $H$-stabilizer of the common boundary loop).  We perform the amalgamated free product of  $H_{i}$ and $H_{j}$ along the trivial or cyclic group $H_{i} \cap H_{j}$. Set $S_{ij}$ be the union of $S_{i}$, $S_{j}$ with the common boundary loop in $S^{*}$ and set $H_{ij}$ the constructed group. Now, if $S_{k}$ is another of the surfaces which has a common boundary loop in $S^{*}$ with $S_{ij}$, then we again perform the amalgamated free product of  $H_{ij}$ and $H_{k}$ along the trivial or cyclic group $H_{ij} \cap H_{k}$.  Continuing with this procedure, we end with a group $H^{*}$ obtained as amalgamated free product along trivial or finite cyclic groups.  If $\alpha$ is any of the boundary loops of $S^{*}$, there should be 
a (not necessarily different) boundary loop $\beta$ of $S^{*}$ and an element $h \in H\setminus\{I\}$ so that $h(\alpha)=\beta$. By the choice of the surfaces $S_{j}$, we must have that $h(S^{*}) \cap S^{*}=\emptyset$. In particular, if $\beta = \alpha$, then $h$ has order two with two fixed points on $\alpha$. Also, if there is another boundary loop $\gamma$ of $S^{*}$ (different from $\beta$) and an element $u \in H$ so that $u(\alpha)=\gamma$, then $hu^{-1} \in H \setminus \{I\}$ sends the region $S_{j}$ containing  $\gamma$ in its boundary to the region $S_{i}$ containing $\beta$ in its boundary, a contradiction to the choice of the regions $S_{k}$.
 We may now perform the HHN-extension of $H^{*}$ by the finite cyclic group generated by $h$. If $\alpha_{1}=\alpha$,\ldots, $\alpha_{m}$ are the boundary loops of $S^{*}$, which are not $H$-equivalent, then we perform the HHN-extension with each of them. At the end, we obtain an isomorphic copy of $H$.
\end{rema}

\subsection{Klein-Maskit's combination theorem}
 Let $K$ be a Kleinian group with non-empty region of discontinuity $\Omega$. Let $H$ be a subgroup of $K$ with limit set $\Lambda(H)$. 
A set $X \subset \widehat{\mathbb C}$ is called {\it precisely invariant under $H$ in $K$} if $U(X)=X$, for every $U \in H$, and $V(X) \cap X = \emptyset$, for every $V \in K\setminus H$. For our purposes, the group $H$ will be either the trivial group, or a finite cyclic group or an infinite cyclic group generated by a parabolic transformation.
If $H$ is a cyclic group, a {\it precisely invariant disc} $B$ is the interior of a closed topological
disc $\overline{B}$, where $\overline{B}-\Lambda(H) \subset \Omega$ is precisely invariant under $H$ in $K$.

\begin{theo}[Klein-Maskit's combination theorem \cite{Maskit:Comb, Maskit:Comb4}]\label{KMC}
\mbox{}

\noindent
(1) ({\bf Amalgamated free products}).
For $j=1,2$, let $K_{j}$ be a Kleinian group, let $H \leq K_{1} \cap K_{2}$ be a cyclic subgroup (either trivial, finite or generated by a parabolic transformation), $H \neq K_{j}$, and 
let $B_{j}$ be a precisely invariant disc under $H$ in $K_{j}$. Assume that $B_{1}$ and $B_{2}$ have as a common boundary the simple loop $\Sigma$ and that $B_{1} \cap B_{2}=\emptyset$.
Then $K=\langle K_{1},K_{2}\rangle$  is a Kleinian group isomorphic to the free product of $K_{1}$ and $K_{2}$ amalgamated over $H$, that is, $K=K_{1} *_{H} K_{2}$, and every elliptic or parabolic element of $K$ is conjugated in $K$ to an element of either $K_{1}$ or $K_{2}$. (In particular, if both $K_{j}$ have no parabolic elements, then neither does $K$.)
Moreover, if $K_{1}$ and $K_{2}$ are both geometrically finite, then $K$ is also geometrically finite. Also, if the limit sets of both $K_{1}$ and $K_{2}$ are totally disconnected, then the same holds for the limit set of $K$.

\noindent
(2) ({\bf HNN extensions}).
Let $K$ be a Kleinian group. For $j=1,2$, let $B_{j}$ be a precisely invariant disc under the cyclic subgroup $H_{j}$  (either trivial, finite or generated by a parabolic)  in $K$, let $\Sigma_{j}$ be the boundary loop of $B_{j}$ and assume that $T(\overline{B}_{1}) \cap \overline{B}_{2} =\emptyset$, for every $T \in K$. Let $A$ a loxodromic transformation such that $A(\Sigma_{1})=\Sigma_{2}$, $A(B_{1}) \cap B_{2}=\emptyset$, and $A^{-1} H_{2} A=H_{1}$. Then $K_{A}=\langle K, A\rangle$ is a Kleinian group, isomorphic to the HNN-extension $K*_{\langle A \rangle}$  (that is, every relation in $K_{A}$ is consequence of the realtions in $K$ and the relations $A^{-1}H_{2}A=H_{1}$). If each $H_{j}$, for $j=1,2$, is its own normalization in $K$, then every elliptic or parabolic element of $K_{A}$ is conjugated to some element of $K$. (In particular, if $K$ has no parabolic elements, then neither does $K_{A}$.) Moreover, if $K$ is geometrically finite, then $K_{A}$ is also geometrically finite. Also, if the limit set of $K$ is totally disconnected, then the same holds for the limit set of $K_{A}$.

\end{theo}

\subsection{Maskit's decomposition of (extended) function groups}\label{Sec:function}
A  finitely generated (extended) Kleinian group $K$ is called an {\it (extended) function group} if it has a $G$-invariant connected component $\Delta$ of its region of discontinuity $\Omega$. Note that (Extended) virtual Schottky groups are examples of (extended) function groups.


Basic examples of function groups are provided by:
(i) {\it elementary  groups}, that is, Kleinian groups with finite limit set (so of cardinality at most $2$); 
(ii) {\it quasifuchsian groups}, that is, function groups whose limit set is a Jordan curve (so its region of discontinuity consists of two invariant connected components); (iii) {\it totally degenerate groups}, that is, non-elementary finitely generated Kleinian  groups whose region of discontinuity is both connected and simply-connected. Similarly, basic examples of extended function groups are provided by:
(iv) {\it extended elementary groups}, that is, extended Kleinian groups with finite limit set; (v) {\it extended quasifuchsian groups}, that is, finitely generated extended function groups whose limit set is a Jordan curve;
(vi) {\it extended totally degenerate groups}, that is,  non-elementary extended finitely generated Kleinian groups with connected and simply-connected region of discontinuity. 

Maskit obtained the following geometrical decomposition picture, in terms of Klein-Maskit's combination theorems, of function groups.

\begin{theo}[Structure of (extended) function groups \cite{Maskit:construction, Maskit:decomposition, Maskit:function1, Maskit:function, H:extendedfunction}]\label{function}
Every function group is constructed from a finite collection of elementary groups, quasifuchsian groups and totally degenerate groups by a finite number of applications of the Klein-Maskit combination theorems. Similarly, every extended function group is constructed from (extended) elementary groups, (extended) quasifuchsian groups and (extended) totally degenerate groups by a finite number of applications of the Klein-Maskit combination theorems. 
\end{theo}

\section{Structure description of (extended) virtual Schottky groups} \label{Sec:estructura}
\subsection{On Schottky groups}
 If $G$ is a Schottky group of rank $g \geq 1$, then it is a purely loxodromic function group with a totally disconnected limit set.
Conversely, Theorem \ref{function} asserts that this provides tan equivalent definition of a Schottky group (this is originalyy due to Maskit).

\begin{coro}\label{coroSchottky}
A non-finite function group is a Schottky group if and only if it has a totally disconnected limit set and it is purely loxodromic.
\end{coro}

\subsection{On (extended) virtual Schottky groups}
Theorem \ref{function} also has the following generalization of Corollary \ref{coroSchottky} at the level of  (extended) virtual Schottky groups. As finite groups are (extended) virtual Schottky group, we only need to take care of the non-finite (extended) function groups.

\begin{prop}\label{corovirtual}
A non-finite (extended) function group is an (extended) virtual Schottky group if and only if it has a totally disconnected limit set and it has no parabolic elements. 
\end{prop}
\begin{proof}
One direction is clear, if $K$ is a non-finite (extended) virtual Schottky group, then it contains, as a finite index subgroup, a Schottky group $G$ of rank $g \geq 1$. The finite index condition asserts that they have the same limits set (so totally disconnected) and that $K$ has no parabolic elements.
In the other direction, let $K$ be a  non-finite (extended) function group with  totally disconnected limit set and containing no parabolic elements. If $K^{+}=K\cap {\mathbb M}$, then $K^{+}$ has index at most two in $K$ and, in particular, it is a function group with totally disconnected limit set and without parabolic elements. As a consequence of Selberg's lemma \cite{Selberg}, there is a finite index torsion free normal subgroup $G$ of $K^{+}$. It follows that $G$ is a purely loxodromic function group with totally disconnected limit set, so a Schottky group (by Corollary \ref{coroSchottky}), and $K$ is an (extended) virtual Schotky group.
\end{proof}

In the particular case of (extended) virtual Schottky groups, Theorem \ref{function} can be written as follows.

\begin{prop}\label{principal}
A function group is a virtual Schottky group if and only if it is constructed from finite subgroups of ${\mathbb M}$ and loxodromic cyclic groups by a finite number of applications of Klein-Maskit's combination theorems. Similarly, an extended function group is an extended virtual Schottky group if and only if it is constructed from finite subgroups of $\widehat{\mathbb M}$, loxodromic/pseudo-hyperbolic cyclic groups by a finite number of applications of Klein-Maskit's combination theorems.
\end{prop}
\begin{proof}
(1) If $K$ is an (extended) function group constructed, by Klein-Maskit's combination theorems, using the groups as in the theorem, then  either: (i) $K$ is finite, in particular, an (extended) virtual Schottky group or (ii) $K$ is non-finite with totally disconnected limit set and containing no parabolic elements, so an (extended) virtual Schottky group by Corollary \ref{corovirtual}.
(2) If $K$ is a non-finite (extended) virtual Schottky group, then it contains a Schottky group $G$ as a finite index subgroup. In particular, $K$ and $G$ both have the same region of discontinuity $\Omega$, so $K$ is a geometrically finite (extended) function group with a totally disconnected limit set $\Lambda$, and $K$ has no parabolic transformations. It then follows, from Theorem \ref{function}, that $K$ is only constructed using (extended) elementary groups without parabolic elements, that is, 
finite groups and finite index extension of either cyclic loxodromic groups or cyclic pseudo-hyperbolic groups. The finite index extension of cyclic loxodromic groups are either free products of two cyclic groups of order two or a HNN-extension of a finite cyclic group by two elliptics of order two or a HNN-extension of a finite cyclic group by a loxodromic transformation (similar situation happens for finite extensions of cyclic pseudo-hyperbolic groups).
\end{proof}

\section{A geometrical structural description in the abelian case}\label{Sec:casoabeliano}

\subsection{Basic virtual Schottky groups}\label{Sec:basicos}
We first describe some particular examples of virtual Schottky groups $K$ (called the basic ones), which will be used in Theorem \ref{Abel}. 

\s
\noindent
{\bf (B1).} 
The first basic virtual Schottky groups $K$ are (i) the finite abelian subgroups of ${\mathbb M}$ (the trivial Schottky group $G$ as finite index) together the cyclic loxodromic ones (in this case $K=G$). We classify them into three types as follows.

\s
\noindent
{\bf Basic virtual Schottky groups of type (T1):} 
Finite cyclic groups. These are conjugated to $\langle E(z)=e^{2 \pi i/n} z \rangle \cong {\mathbb Z}_{n}$, for $n \geq 0$  integer. 

\s
\noindent
{\bf Basic virtual Schottky groups of type (T2):}
Cyclic groups generated by a loxodromic transformation. These are conjugated to $\langle L(z)=\lambda z\rangle \cong {\mathbb Z}$, where $|\lambda|>1$.

\s
\noindent
{\bf Basic virtual Schottky groups of type (T3):}
Groups isomorphic to ${\mathbb Z}_{2}^{2}$. These are conjugated to $\langle U(z)=-z, V(z)=1/z\rangle \cong {\mathbb Z}_{2}^{2}$.

\s
\noindent
{\bf (B2).} 
The second list of basic virtual Schottky groups are obtained from HNN-extensions, in the sense of Klein-Maskit's combination theorem, of a non-trivial finite abelian group by loxodromic transformations.
Let start with a non-trivial finite abelian group $K^{*}<{\mathbb M}$ (so it is of either type (T1) or (T3)). Let $U_{1},\ldots, U_{n} \in K\setminus \{I\}$ be such that: (i) $\langle U_{j} \rangle$ is not a proper subgroup of another cyclic subgroup of $K^{*}$ and (ii)  $\langle U_{1} \rangle, \ldots, \langle U_{n} \rangle$ are different. So: (i) $n=1$ for $K^{*}$ of type (T1), and (ii) $n \in \{1,2,3\}$ for $K^{*}$ of type (T3).

We now consider pairwise disjoint simple loops $W_{1,j}, W_{2,j}$, each one invariant under $U_{j}$ (in particular, each one separates the two fixed points of $U_{j}$). This choice is also made in order that all the $K^{*}$-translates of them produces a pairwise disjoint collection of loops. Choose a loxodromic transformation $A_{j}$ such that $A(W_{1,j})=W_{2,j}$ (sending the disc bounded by $W_{1,j}$ and disjoint from $W_{2,j}$ to the complement of the disc bounded by $W_{2,j}$ not containing $W_{1,j}$) and which commutes with $U_{j}$. (This is always possible if these loops are chosen to be circles.) 
Using such loxodromic transformations, we may consider the group $K=\langle K^{*},A_{1},\ldots,A_{n}\rangle$. By Klein-Maskit's combination theorem,  $K$ is an HNN-extension of $K^{*}$ by the elements $A_{1},\ldots, A_{n}$. It can be seen that the group $G=\langle 
TA_{j}T^{-1}: T \in K^{*}\rangle$ is a Schottky group  which is a normal finite index subgroup of $K$ such that $K/G \cong {\mathbb Z}_{2}^{2}$. The groups we obtain, in this type of construction, are the following ones.

\s
\noindent
{\bf Basic virtual Schottky groups of type (T4):}
Groups conjugated to $K=\langle  A(z)=\lambda z, E(z)=e^{2 \pi i/n}z\rangle \cong {\mathbb Z} \times {\mathbb Z}_{n}$, where $|\lambda| >1$. In this case, $G=\langle A \rangle$ and $K/G \cong {\mathbb Z}_{n}$.

\s
\noindent
{\bf Basic virtual Schottky groups of type (T5):}
Groups conjugated to $K=\langle U(z)=-z, V(z)=1/z, A(z)=\lambda z\rangle$, where $|\lambda|>1$.
In this case, $G=\langle A \rangle$ and $K/G \cong {\mathbb Z}_{2}^{2}$.

\s
\noindent
{\bf Basic virtual Schottky groups of type (T6):}
Groups conjugated to $K=\langle U(z)=-z, V(z)=1/z, A(z)=\lambda_{1} z, B(z)=((\lambda_{2}+1)z+(1-\lambda_{2}))/((1-\lambda_{2})z+(\lambda_{2}+1)) \rangle $, where $|\lambda_{j}|>1$, (so $AU=UA$ and $BV=VB$). 
In this case, $G=\langle A,B\rangle$ (a Schottky group of rank two) and $K/G \cong {\mathbb Z}_{2}^{2}$.

\s
\noindent
{\bf Basic virtual Schottky groups of type (T7):}
Groups conjugated to $K=\langle U(z)=-z, V(z)=1/z, A(z)=\lambda z, B(z)=((\lambda_{2}+1)z+(1-\lambda_{2}))/((1-\lambda_{2})z+(\lambda_{2}+1)), C(z)=((\lambda_{3}+1)z+i(1-\lambda_{3}))/(i(\lambda_{3}-1)z+(\lambda_{3}+1))\rangle$, where $|\lambda_{j}|>1$, (so $AU=UA$, $BV=VB$ and $CUV=UVC$). In this case, $G=\langle A, B, C\rangle$ (a Schottky group of rank three) and $K/G \cong {\mathbb Z}_{2}^{2}$.

\s
\noindent
{\bf (B3).} The third type of basic virtual Schottky groups are obtained as an amalgamated free product of some finite number of copies of groups of types (T3), (T5) and/or (T6). If $K_{1}$ and $K_{2}$ are two of them such that $K_{1} \cap K_{2}=\langle U_{1}\rangle \cong {\mathbb Z}_{2}$, then we may perform the free amalgamated products $K_{12}:=K_{1} *_{\langle U_{1} \rangle} K_{2}$ (in the sense of Klein-Maskit's combination theorem). Next, we consider a third one $K_{3}$ such that $K_{12} \cap K_{3}=\langle U_{2}\rangle \cong {\mathbb Z}_{2}$, where $U_{2} \neq U_{1}$. Then we again perform the free amalgamated products $K_{123}:=K_{12} *_{\langle U_{2} \rangle} K_{3}$. We continue with this process a finite number of times.

\begin{rema}
If $K$ is a basic virtual Schottky group, and its region of discontinuity is $\Omega$, then the orbifold $\Omega/K$ is:
(i) of signature $(0;n,n)$ for (T1), (ii) torus for (T2) and (T4), (iii)  of signature $(0;2,2,2)$ for type (T3), (iv) of signature $(0;2,2,2,2)$ for type (T5), (v) of signature $(0;2,2,2,2,2)$ for type (T6), (vi) of signature $(0;2,2,2,2,2,2)$ for type (T7) and (vii) of signature $(0;2,\ldots,2)$ for type (B3).
\end{rema}

\subsection{Main result}
The following states an structural decomposition of those virtual Schottky groups $K$, admitting a Schottky group $G$ as a finite index normal subgroup such that $K/G$ is an abelian group (which generalizes the one obtained in \cite{Hidalgo:Cyclic}).

\begin{theo}\label{Abel}
{\rm (1)} A Kleinian group $K$ constructed as a free group, in the sense of Klein-Maskit's combination theorem, of basic virtual Schottky groups is a virtual Schottky group containing a Schottky group $G$, as a finite index normal subgroup and such that $K/G$ is an abelian group.
{\rm (2)} A virtual Schottky group $K$, containing a Schottky group $G$
as a finite index normal subgroup and such that $K/G$ is an abelian group, is the free product, in the sense of Klein-Maskit's combination theorem, of a finite collection of basic virtual Schottky groups.
\end{theo}

\begin{rema}
If $K$ is constructed as free product, in the sense of Klein-Maskit's theorem, using the basic virtual Schottky groups, then the above theorem asserts that it is a virtual Schottky group and that it must have a Schottky group $G$ as a finite index normal subgroup such that $K/G$ is an abelian group. To construct explicitly such a Schottky group is, in general, not so easy to do. Another interesting question is to determine how many different, up to $K$-conjugation, such Schottky subgroups of minimal rank does $K$ have.
\end{rema}

\subsection{Example: structural description in the cyclic case}\label{Sec:casociclico}
In the particular case that $H=K/G \cong {\mathbb Z}_{n}$, where $n \geq 2$, $K$ is a virtual Schottky group and $G$ is a Schottky group, being a finite index normal subgroup of $K$, one may see that in Theorem \ref{Abel} the only groups to be used are of types (T1), (T2) and (T4) (see Figure \ref{dominio}). This, in particular, provides the description in \cite{Hidalgo:Cyclic}, which states that  
there are integers
$$\left\{\begin{array}{l}
a,b,c,d \in \{0,1,2,...\},\\
n_{1},...,n_{d} \in \{3,...,n\},\;
m_{1},...,m_{b} \in \{2,...,n\},\;
\mbox{$n_{j}$ and $m_{j}$ divisors of $n$},
\end{array}
\right.
$$
satisfying 
$g=n(a+b+c/2+d-1)+1 -n\sum_{j=1}^{d}1/n_{j},$
and either 
\begin{enumerate}
\item $a+b>0$; or
\item $a=b=0$, $c>0$ and $GCD(n/2,n/n_{1},...,n/n_{d})=1$; or
\item $a=b=c=0$ and $GCD(n/n_{1},...,n/n_{d})=1$,
\end{enumerate}
there are  loxodromic transformations $\tau_{1},..., \tau_{a}, \eta_{1},..., \eta_{b} \in K,$
and elliptic transformations
$\theta_{1},...,\theta_{b},\gamma_{1},...,\gamma_{c},\epsilon_{1},...,\epsilon_{d} \in K,$
such that 
\begin{itemize}
\item[(i)] the order of $\theta_{j}$ is $m_{j}$;
\item[(ii)] the order of $\gamma_{j}$ is $2$ (they only appear in the case $n$ is even);
\item[(iii)] the order of $\epsilon_{j}$ is $n_{j}$; and 
\item[(iv)] $\eta_{j} \circ \theta_{j} = \theta_{j} \circ \eta_{j}$ commute,
\end{itemize}
and there is a collection simple loops as shown in figure \ref{dominio}, such that $K$ is the free product (in the sense of the Klein-Maskit combination theorem) of the ``$a$" cyclic loxodromic groups $\langle \tau_{j} \rangle$, the ``$b$" cyclic groups of order two $\langle \gamma_{j}\rangle$, the ``$d$" cyclic elliptic groups $\langle \epsilon_{j}\rangle$ and the ``$b$" abelian groups $\langle \eta_{j}, \theta_{j}\rangle$, that is, 
$$K \cong {\mathbb Z} *\stackrel{a}{\cdots}* {\mathbb Z}
*({\mathbb Z} \oplus {\mathbb Z}_{m_{1}}) * \cdots
*({\mathbb Z} \oplus {\mathbb Z}_{m_{b}})*
{\mathbb Z}_{2} *\stackrel{c}{\cdots}* {\mathbb Z}_{2}*
{\mathbb Z}_{n_{1}} * \cdots * {\mathbb Z}_{n_{d}}.$$

\begin{figure}[ht]
\centering
\includegraphics[width=6.0cm]{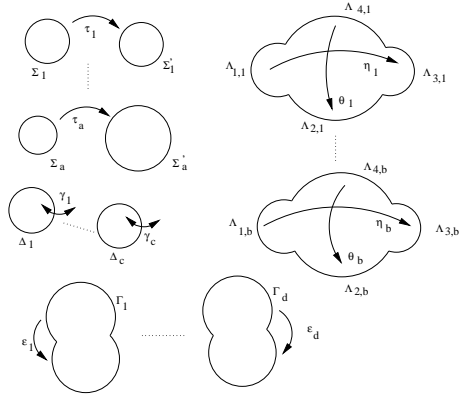}
\caption{The structural picture of $K$ for the cyclic case}
\label{dominio}
\end{figure}

\section{Proof of Theorem \ref{Abel}}\label{Sec:prueba}
\subsection{Proof of part (1)} A Kleinian group $K$ constructed as the free product, in the sense of Klein-Maskit's combination theorem, of basic virtual Schottky groups  is a function group with totally disconnected limit set and such that its non-loxodromic transformation are elliptic of finite order. It follows from Corollary \ref{corovirtual} that $K$ is a virtual Schottky group.

\subsection{Proof of part (2)} Let $K$ be a virtual Schottky group containing a Schottky group $G$, as a finite index normal subgroup and such that $H=K/G$ is an abelian group. If $G$ has rak $g \in \{0,1\}$ (so $K$ is elementary), then $K$ is one of the types (T1)-(T7). So, let us assume, from now on, that $G$ has rank $g \geq 2$ (that is, $K$ is non-elementary).

Set  $S=\Omega/G$ and let us fix a Schottky uniformization $(\Omega,G,P:\Omega \to S)$. As $G$ is normal subgroup of $K$, we have that $H$ lifts, with respect to $P$, to $K$.  Now, Theorem \ref{teo:lift} ensures the existence of a Schottky system of loops for $H$, say ${\mathcal F}=\{L_{1},...., L_{k}\}$ ($g \leq k \leq 3g-3$), corresponding to the above Schottky uniformization. We assume ${\mathcal F}$ to be minimal, that is, no proper sub-collection is a Schottky system of loops for $H$ corresponding to the above Schottky uniformization.

The covering map $P:\Omega \to S$ induces a surjective homomorphism $\Theta:K \to H$, whose kernel is  $G$, satisfying that $\Theta(k) \circ P= P \circ k$, for all  $k \in K$.

\subsubsection{Structure loops and structure regions}
As each loop in $\mathcal F$ lifts to loops under $P:\Omega \to S$, we may consider the collection $\widehat{\mathcal F} \subset \Omega$ of these lifted loops. A loop in $\widehat{\mathcal F}$ will be called an {\it structure loop} and each component of $\Omega \setminus \widehat{\mathcal F}$ an {\it structure region}. 

 Note that the minimality property of ${\mathcal F}$ ensures that each structural region must contain at least three boundary loops.

\subsubsection{$K$-stabilizers of structure loops and structure regions}
Let $W \in \widehat{\mathcal F}$ be an structure loop and let $\widehat{R} \subset \Omega \setminus \widehat{\mathcal F}$ be an structure region. Set $L=P(W) \in {\mathcal F}$ and $R=P(\widehat{R}) \subset S \setminus {\mathcal F}$.
We consider the corresponding stabilizers
$$
K(W)=\{U \in K: U(W)=W\}<K,\quad
H(L)=\{h \in H: h(L)=L\}<H,
$$
$$
K(\widehat{R})=\{U \in K: U(\widehat{R})=\widehat{R}\}<K,\quad
H(R)=\{h \in H: h(R)=R\}<H.
$$

 It follows that   $\Theta(K(W))=H(L)$ and $\Theta(K(\widehat{R}))=H(R)$.

As $G$ is torsion free, both restrictions $P:W \to L$ and $P:\widehat{R} \to R$ are homeomorphisms and both homomorphisms
$\Theta:K(W) \to H(L)$ and  $\Theta:K(\widehat{R}) \to H(R)$ are isomorphisms.

This, asserts that $K(\widehat{R})$ and $K(W)$ are finite abelian groups. As the finite abelian subgroups of ${\rm PSL}_{2}({\mathbb C})$ are either: (i) the trivial group, (ii) cyclic groups or (iii) isomorphic to ${\mathbb Z}_{2}^{2}$, the following lemma follows.

\begin{lemm}\label{lema2}
{\rm (1)} Let $W \in \widehat{\mathcal F}$ and let $\widehat{R}_{1}$ and $\widehat{R}_{2}$ be the two structure regions containing $W$ on their boundaries. Then $K(W)$ is either: (i) trivial, (ii) cyclic or (iii) isomorphic to ${\mathbb Z}_{2}^{2}$. Moreover, {\rm (1.1)} if $K(W)=\langle E \rangle \cong {\mathbb Z}_{n}$, $n \geq 2$, then either: 
(a) both fixed points of $E$ are separated by $W$ or (b) $n=2$ and $E$ has both fixed points on $W$ (so it permutes $\widehat{R}_{1}$ with $\widehat{R}_{2}$), and  {\rm (1.2)} if $K(W)\cong {\mathbb Z}_{2}^{2}$, then exactly one of the three elliptic elements of order two keeps invariant each $\widehat{R}_{j}$ and the others two permutes them (so they have their fixed points on $W$). 

 {\rm (2)} If $\widehat{R}$ is an structure region, then $K(\widehat{R})$ is either: (i) trivial, (ii) cyclic or (iii) isomorphic to ${\mathbb Z}_{2}^{2}$.
\end{lemm}

\begin{rema}
Let $\widehat{R}$ be an structure region and $W$ be an structure loop on its boundary. Either the $K(\widehat{R})$-stabilizer of $W$ is trivial or a cyclic group. Let us assume this stabilizer to be a non-trivial group, say $K_{0}=\langle U \rangle \leq K(\widehat{R})$. Then $K_{0}$ is a maximal cyclic subgroup of $K(\widehat{R})$ and both fixed points of $U$ are separated by $W$. Either: (i) $K_{0}=K(W)$ or (ii) $K(W) =\langle U,V\rangle \cong {\mathbb Z}_{2}^{2}$, where $V$ has its both fixed points on $W$.
\end{rema} 

We also have the following fact.

\begin{lemm}\label{lema5.2}
Let $\widehat{R}$ be an structure region with $K(\widehat{R})=\langle U, V \rangle \cong {\mathbb Z}_{2}^{2}$. Then there is no $T \in K$ such that $TUT^{-1}=V$.
\end{lemm}
\begin{proof}
This follows from the fact that $K/G$ is an abelian group.
\end{proof}

Let $\widehat{R}$ be an structure region and $U \in K(\widehat{R})\setminus\{I\}$. 
The group $K$ is a geometrically finite Kleinian group, containing no parabolic transformations. In \cite{H:mefp} it was observed that either: (i) both fixed points of $U$ belong to $\Omega$, or (ii)  or there is a loxodromic transformation $A \in K$ such that $AU=UA$. In the next proposition  we observe that if $U$ has one of its fixed points in $\widehat{R}$, then the same holds for its other fixed point. 
 
\begin{prop}\label{fixedpoints}
Let $\widehat{R}$ be an structure region and $U \in K(\widehat{R}) \setminus \{I\}$. If one of the fixed points of $U$ belongs to $\widehat{R}$, then does the other fixed point.
\end{prop}
\begin{proof}
Without loss of generality, me may assume that either: 
(1) $K(\widehat{R})=\langle U \rangle \cong {\mathbb Z}_{n}$, where $n \geq 2$, or 
(2) $K(\widehat{R})=\langle U,V \rangle \cong {\mathbb Z}_{2}^{2}$.
If we are in case (2), then the involution $V$ permutes both fixed points of $U$, and we are done. 

Let us now consider the case (1). Let us assume, by the contrary, that $U$ has exactly one of its fixed points on $\widehat{R}$. Then there is one (and only one) structure loop $W$, on the boundary of $\widehat{R}$, which is stabilized by $U$. Moreover, any of the other boundary loops has trivial $K(\widehat{R})$-stabilizers. This, in particular, asserts that $P(W) \in {\mathcal F}$ cannot be $H$-equivalent to any of the other boundary loops of $P(\widehat{R})$.

Let $\widehat{R}_{1}$ be the other structure region with $W$ on its boundary. 
These two structure regions cannot be equivalent under $G$. In fact, if there is some $T \in G\setminus \{I\}$ such that $T(\widehat{R}_{1})=\widehat{R}$, then (as $\widehat{R}_{1}$ is stabilized by $U$) $T(W)$ must be a boundary loop of $\widehat{R}$ which has nontrivial  $K(\widehat{R})$-stabilizer. So, as previously noted, $T(W)=W$, which means that $T \in K(W)$, a contradiction as $T$ has infinite order and $K(W)$ is finite. Now, this asserts that $P(\widehat{R})$ and $P(\widehat{R}_{1})$ are different and that $P({\widehat R} \cup W \cup {\widehat R}_{1})$ is also planar.

(1.1) Assume that $K(\widehat{R}_{1})=\langle U \rangle$.

If the other fixed point of $U$ belongs to $\widehat{R}_{1}$, then (as its other boundary loops have trivial $K(\widehat{R})$-stabilizer) $W$ cannot be $K$-invariant to any other boundary loop of $\widehat{R}_{1}$.
If $\widehat{R}_{1}$ has not a fixed point of $U$, then it has a boundary loop $W_{1} \neq W$ which is invariant under $U$. Assume there is some $T \in K$ such that $T(W_{1})=W$. Then $T(W)$ is a boundary loop of $\widehat{R}$, different from $W$ which is invariant under $U$, a contradiction. 
So, in this situation neither the loop $W$ can be $K$-equivalent to any other boundary loop of $\widehat{R}_{1}$.

We have proved that the none of the boundary loops of $\widehat{R}_{1} \cup W \cup \widehat{R}$ is $K$-equivalent to $W$. So, we may delete the $H$-translates of $P(W)$ from ${\mathcal F}$ and still having a Schottky system of loops for $H$, and we get a contradiction to the minimality of ${\mathcal F}$.

(1.2) Assume that $n=2$ and $K(\widehat{R}_{1})=\langle U,V\rangle \cong {\mathbb Z}_{2}$.

 In this case, $W_{1}=V(W)$ is a boundary loop of $\widehat{R}_{1}$ also invariant under $U$. The region
 $\widehat{R}_{2}:=V(\widehat{R})$ is the other structural region sharing $W_{1}$ in its boundary. 
 
We may note that $W$ (and so $W_{1}$) are non-$K$-equivalent to any boundary loop of the region $X:=\widehat{R} \cup W \cup \widehat{R}_{1} \cup W_{1} \cup \widehat{R}_{2}$. This is clear for those boundary loops of $X$ which are also boundary loops of either $\widehat{R}$ and $\widehat{R}_{2}$ (if there is some $T \in K$ such that $T(W)$ is a boundary loop of either $\widehat{R}$ or $\widehat{R}_{2}$, different from $W$ and $W_{1}$, then $TUT^{-1}$ will be stabilizing $T(W)$ and also the corresponding region, a contradiction). The same holds for those boundary loops of $\widehat{R}_{1}$ with trivial $K(\widehat{R}_{1}$-stabilizer. Lemma \ref{lema5.2} takes care of the other boundary loops.

So, we may delete the $H$-translates of $\{P(W),P(W_{1})\}$ from ${\mathcal F}$ and still having a Schottky system of loops for $H$, and we get a contradiction to the minimality of ${\mathcal F}$.
\end{proof}

In the above proposition we have seen that if $U \in K(\widehat{R}) \setminus \{I\}$ has one fixed point on $\widehat{R}$, then the other fixed point also does. In the next proposition we consider the case when both fixed points of $U$ do no belong to such a region.

\begin{prop}\label{propo3}
Let $\widehat{R}$ be an structure region with either 
{\rm (1)} $K(\widehat{R})=\langle U \rangle \cong {\mathbb Z}_{n}$, where $n \geq 2$, or {\rm (2)} $K(\widehat{R})=\langle U,V \rangle \cong {\mathbb Z}_{2}^{2}$. If 
none of the two fixed points of $U$ are in $\widehat R$, then there are two different boundary loops of $\widehat R$, say $W_{1}$ and $W_{2}$, each one invariant under $U$. In case {\rm (1)} there is a loxodromic transformation $L \in K$ commuting with $U$ and such that $L(W_{1})=W_{2}$. Moreover, in this case, $K(W_{j})=\langle U\rangle$.
In case {\rm (2)} either: (i) there is a loxodromic transformation $L \in K$ commuting with $U$ and such that $L(W_{1})=W_{2}$, in which case, $K(W_{j})\cong {\mathbb Z}_{2}^{2}$,
or (ii) $K(W_{j})=\langle U\rangle$, the $K$-stabilizer of the other structural region sharing $W_{j}$ in its boundary is isomorphic to ${\mathbb Z}_{2}^{2}$ and both stabilizers intersect exactly on $K(W_{j})$.

\end{prop}

\begin{proof}
As both fixed points of $U$ are not contained on $\widehat R$, it follows that there are two different boundary structure loops on it, say $W_{1}$ and $W_{2}$, each one invariant under $U$. Recall that $\widehat R$ must contains at least three boundary loops (by the minimality of the collection ${\mathcal F}$).

(1) Let $K(\widehat{R})=\langle U \rangle \cong {\mathbb Z}_{n}$, where $n \geq 2$. By Lemma \ref{lema2}, the $K(\widehat{R})$-stabilizer of each $W_{1}$ and $W_{2}$ is $\langle U\rangle$ and the other boundary structure loops of $\widehat{R}$ have trivial $K(\widehat{R})$-stabilizer. 

We need to prove that there is a loxodromic transformation $L \in K$ such that $L(W_{1})=W_{2}$. Note that, as there is no element of  $K(\widehat{R})$ sending $W_{1}$ to $W_{2}$, if we obtain some $L \in K$ such that $L(W_{1})=W_{2}$, then $L$ cannot have finite order, so it must be a loxodromic transformation.

Let us assume, by the contrary, that there is no element of $K$ sending $W_{1}$ to $W_{2}$. This assumption ensures that $P(W)$ cannot be $H$-equivalent to the other boundary loops of $P(\widehat{R})$. Now, we may proceed similarly as in the proof of proposition \ref{fixedpoints}.

Let $\widehat{R}_{1}$ be the other structure region containing $W_{1}$ as a boundary loop. 
Then $\widehat{R}$ and $\widehat{R}_{1}$ are non-equivalent under $G$. To see this, assume there is some $A \in G\setminus\{I\}$ such that $A(\widehat{R}_{1})=\widehat{R}$. Then $A(W_{1})$ is a boundary structure loop of $\widehat{R}$ which is stabilized by $A U A^{-1}$, which also stabilizes $\widehat{R}$ (as $U$ stabilizes $\widehat{R}_{1}$).
The only possibilities are $A(W_{1}) \in \{W_{1},W_{2}\}$. By our assumption, we only may have $A(W_{1})=W_{1}$, that is, $A \in K(W_{1})$. This is a contradiction as $A$ has infinite order and $K(W_{1})$ is finite. Now, this asserts that $P({\widehat R})$ and $P({\widehat R}_{1})$ are different and that $P({\widehat R} \cup W_{1} \cup {\widehat R}_{1})$  is still planar.

In the structural region $\widehat{R}_{1}$ there is another boundary loop $W_{3} \neq W_{1}$ which is invariant under $U$. 

Assume that $K(\widehat{R}_{1})=\langle U \rangle$.
If there is some $T \in K$ such that $T(W_{3})=W_{1}$, then $T$ should be loxodromic and $T(W_{1})=W_{2}$, a contradiction to our assumption. Similarly as in the proof of Proposition \ref{fixedpoints}, we see that the loop $W_{1}$ is non-$K$-equivalent to any other boundary loop of ${\widehat R} \cup W_{1} \cup {\widehat R}_{1}$. So, 
we may delete the $H$-translates of $P(W)$ from ${\mathcal F}$ to still having a Schottky system of loops for $H$, a contradiction to the minimality of ${\mathcal F}$.

Assume that $K(\widehat{R}_{1})=\langle U,V \rangle \cong {\mathbb Z}_{2}^{2}$, so $n=2$. In this case, $V(W_{1})=W_{3}$. We let $\widehat{R}_{2}=V(\widehat{R})$ and consider the region $X:=\widehat{R} \cup W_{1} \cup \widehat{R}_{1} \cup W_{3} \cup \widehat{R}_{2}$. We need to observe that the loops $W_{1}$ and $W_{3}$ cannot be $K$-equivalent to any other loops in the boundary of $X$. In fact, if there is some $T \in K$ such that $T(W_{1})=V(W_{3})$, then $VT(W_{1})=W_{2}$, a contradiction. If there is $T \in K$ such that $T(W_{1})$ is a boundary loop of $X$, also in the boundary of $\widehat{R}_{1}$, then $T(\widehat{R})=\widehat{R}_{1}$, a contradiction as these two regions have different $K$-stabilizers. So, we may  delete the $H$-translates of $\{P(W_{1}),P(W_{3})\}$ from ${\mathcal F}$ to still having a Schottky system of loops for $H$, a contradiction to the minimality of ${\mathcal F}$.  

Now, all the above asserts the existence of the loxodromic element $L \in K$ such that $L(W_{1})=W_{2}$. This ensures that $L$ 
conjugates the $K$-stabilizer of $W_{1}$ onto the $K$-stabilizer of $W_{2}$, that is, $L$ normalizes $\langle U\rangle$. It follows that $L$ and $U$ must have the same fixed points (so they commute).

Note that $K(W_{j}) \cap K(\widehat{R})=\langle U\rangle$.
By Lemma \ref{lema2}, if $K(W_{j}) \neq \langle U \rangle$, then $n=2$ and $K(W_{j})=\langle U, V\rangle \cong {\mathbb Z}_{2}^{2}$, where both fixed points of $V$ are on $W_{j}$. But in this case, it is possible to observe that $LV \in K(\widehat{R}) \setminus \langle U \rangle$, a contradiction.

(2) Let $K(\widehat{R})=\langle U,V \rangle \cong {\mathbb Z}_{2}^{2}$. If $W_{1}$ is invariant under an elliptic transformation of order two $E \in K$, with both fixed points on it, then $L=VE$ is a loxodromic transformation such that $L(W_{1})=W_{2}$ (and commuting with $U$). Similarly, if we replace $W_{1}$ by $W_{2}$. 

Let us assume that none of $W_{1}$ and $W_{2}$ is invariant under under such types of involutions in $K$ and that 
there is not a loxodromic transformation $L$ such that $L(W_{1})=W_{2}$. If $\widehat{R}_{1}$ is the other structural region sharing $W_{1}$ in its boundary, then its $K$-stabilizer contains $U$. Either $K(\widehat{R}_{1})=\langle U \rangle$ or $K(\widehat{R}_{1})=\langle U,F \rangle \cong {\mathbb Z}_{2}^{2}$.  The cyclic situation cannot happen by (1) above. In the second situation, the loxodromic transformation $VF$ sends the boundary loop $F(W_{1})$ to $W_{2}$.

\end{proof}

\begin{coro}\label{coro4}
Let $W \in \widehat{\mathcal F}$ be such that $K(W) \neq \{I\}$ and let $\widehat{R}_{1}$ and $\widehat{R}_{2}$ be the two structural regions sharing $W$ in their boundaries. Either (i) these two regions are $K$-equivalent or (ii) $K(\widehat{R}_{1}) \cong {\mathbb Z}_{2}^{2}$ and $K(W)=K(\widehat{R}_{1}) \cap K(\widehat{R}_{1})=\langle U \rangle \cong {\mathbb Z}_{2}$.
\end{coro}

\begin{rema}\label{observa1}
(I) If $K(\widehat{R})=\langle U \rangle \cong {\mathbb Z}_{n}$, where the two fixed points of $U$ are not in $\widehat{R}$, then part (1) of Theorem \ref{propo3} asserts the existence of two structural loops $W_{1}, W_{2}$ in its boundary, each one invariant under $U$, and a loxodromic element $L \in G$ such that $L(W_{1})=W_{2}$. In this case, the same theorem asserts that 
$K(W_{j})=\langle U\rangle$ and $\langle U,L\rangle$ is a group of type (T4). 
(II) If $K(\widehat{R})=\langle U,V\rangle \cong {\mathbb Z}_{2}^{2}$, then either both fixed points of $C \in \{U,V,UV\}$ belong to $\widehat{R}$ or there are two boundary loops $W_{1,C}, W_{2,C}$ of $\widehat{R}$, each one invariant under $C$. One of the possibilities is that there is a loxodromic element $L_{C} \in K$, commuting with $C$ such that $L_{C}(W_{1,C})=W_{2,C}$, in which case $V_{C}=DL_{C}$ (where $C\neq D \in \{U,V,UV\}$) has order two, stabilizes $W_{1,C}$ and the group generated by $K(\widehat{R})$ and these loxodromic transformations (if any) provides the basic virtual Schottky groups of either types (T3), (T5), (T6) or (T7).  In the other possibility, if $\widehat{R}_{1}$ is the other structural region sharing $W_{1,C}$, then this structural region also has $K$-stabilizer isomorphic to ${\mathbb Z}_{2}^{2}$,
$K(\widehat{R}) \cap K(\widehat{R}_{1})=K(W_{1,C})$ and 
$\langle K(\widehat{R}), K(\widehat{R}_{1})\rangle =K(\widehat{R})*_{K(W_{1,C})} K(\widehat{R}_{1})$. This produces basic virtual Schottky groups of type (B3).

\end{rema}

\subsubsection{A choice of a maximal region}
We may proceed similarly as done in \cite{Hidalgo:Cyclic}.
As $S/H=\Omega/K$ is connected, we may find a maximal finite collection of non-$K$-equivalent structural regions $\widehat{R}_{1},\ldots, \widehat{R}_{n}$ such that, if we denote by $\overline{\widehat{R}}_{j}$ the clousure of $\widehat{R}_{j}$ (this is just adding to it its boundary structure loops), then 
$\widetilde{R}=\overline{\widehat{R}}_{1} \cup \cdots \cup \overline{\widehat{R}}_{n}$ is connected. Let $K_{j}=K(\widehat{R_{j}})$ be the $K$-stabilizer of $\widehat{R}_{j}$. 

If $W \in \widehat{\mathcal F}$ is contained in the interior of $\widetilde{R}$, then Corollary \ref{coro4} asserts that either (i) $K(W)=\{I\}$ or (ii) $K(W)$ coincides with the intersection of the $K$-stabilizers of both structural regions sharing it on the boundary (in this case, these two $K$stabilizers are isomorphic to ${\mathbb Z}_{2}^{2}$ and the group generated by them happens to be a free product amalgamated over $K(W)$). Those structural loops contained on the border of $\widetilde{R}$ may either as trivial $K$-stabilizer or to be a cyclic group of order two generated by an elliptic transformation with both fixed points on it.

Next, for each $j=1,\ldots, n$, we proceed to set some subgroups $K_{j}^{*}<K$ which are extensions of $K_{j}$ (and are as the ones described in the theorem). 

(1) If either $K_{j}$ is trivial or every of it non-trivial element have both fixed points on $\widehat{R}_{j}$, then we set $K_{j}^{*}=K_{k}$ (these are basic virtual Schottky groups of type (T1) and (T3)).

(2) If $K_{j}=\langle U_{j} \rangle \cong {\mathbb Z}_{n_{j}}$, $n_{j} \geq 2$, such that both fixed points of $U_{j}$ do not belong to $\widehat{R}_{j}$, then (by Proposition \ref{propo3}) there there are two structural boundary loops $W_{1}$ and $W_{2}$ of $\widehat{R}_{j}$, each one invariant under $U_{j}$, and there is a loxodromic transformation $A_{j} \in K$ commuting with $U_{j}$ such that $A_{j}(W_{1})=W_{2}$. In this case, by Klein-Maskit's combination theorem, $K_{j}^{*}=\langle K_{j}, A_{j}\rangle=K_{j}*_{\langle A_{j} \rangle} \cong {\mathbb Z}_{n_{j}}*_{\mathbb Z}$ (a basic virtual Schottky group of type (T4)).

(3) Let us assume $K_{j}=\langle U_{j}, V_{j} \rangle \cong {\mathbb Z}_{2}^{2}$ and both fixed points of at least one of its order two elements has not its fixed points on $\widehat{R}_{j}$. 
(3.1) If both fixed points of $U_{j}$ do not belong to $\widehat{R}_{j}$, but the fixed points of $V_{j}$ and $U_{j}V_{j}$ do, then
 there there are two structural boundary loops $W_{1}$ and $W_{2}$ of $\widehat{R}_{j}$, each one invariant under $U_{j}$.
 (By Proposition \ref{propo3}), one possibility is that there is a loxodromic transformation $A_{j} \in K$, commuting with $U_{j}$ and $A_{j}(W_{1})=W_{2}$. In this case, by Klein-Maskit's combination theorem, 
$K_{j}^{*}=\langle K_{j}, A_{j} \rangle=\langle K_{j} \rangle *_{ \langle A_{j} \rangle} \cong {\mathbb Z}^{2}_{2}*_{\mathbb Z}$ (a basic virtual Schottky group of type (T5)).
(3.2) If both fixed points of $U_{j}$ and  of $V_{j}$ do not belong to $\widehat{R}_{j}$, but the fixed points of $U_{j}V_{j}$ do, then one possibility (by Proposition \ref{propo3}) is that there there are structural boundary loops $W_{1}$, $W_{2}$, $W_{3}$ and $W_{4}$ of $\widehat{R}_{j}$, each $W_{1}$ and $W_{2}$ (respectively, $W_{3}$ and $W_{4}$) invariant under $U_{j}$ (respectively, $V_{j}$), and there are loxodromic transformation $A_{j}, B_{j} \in K$, with $A_{j}$ commuting with $U_{j}$ (respectively, $B_{j}$ commuting with $V_{j}$), $A_{j}(W_{1})=W_{2}$ and $B_{j}(W_{3})=W_{4}$. In this case, by Klein-Maskit's combination theorem, 
$K_{j}^{*}=\langle K_{j}, A_{j}, B_{j} \rangle=(\langle K_{j} \rangle *_{\langle A_{j} \rangle}) *_{\langle B_{j} \rangle}\cong ({\mathbb Z}^{2}_{2}*_{\mathbb Z})*_{\mathbb Z}$ (a basic virtual Schottky group of type (T6)).
(3.3) If both fixed points of $U_{j}$, $V_{j}$ and $U_{j}V_{j}$ do not belong to $\widehat{R}_{j}$,  then one possibility (by Proposition \ref{propo3}) is that there there are structural boundary loops $W_{1}$, $W_{2}$, $W_{3}$, $W_{4}$, $W_{5}$ and $W_{6}$ of $\widehat{R}_{j}$, each $W_{1}$ and $W_{2}$ (respectively, $W_{3}$ and $W_{4}$, $W_{5}$ and $W_{6}$) invariant under $U_{j}$ (respectively, $V_{j}$, $U_{j}V_{j}$), and there are loxodromic transformation $A_{j}, B_{j}, C_{j} \in K$, with $A_{j}$ (respectively, $B_{j}$, $C_{j}$) commuting with $U_{j}$ (respectively, $V_{j}$, $U_{j}V_{j}$), $A_{j}(W_{1})=W_{2}$, $B_{j}(W_{3})=W_{4}$ and $C_{j}(W_{5})=W_{6}$. In this case, by Klein-Maskit's combination theorem, 
$K_{j}^{*}=\langle K_{j}, A_{j}, B_{j}, C_{j} \rangle=((\langle K_{j} \rangle *_{\langle A_{j} \rangle}) *_{\langle B_{j} \rangle}) *_{\langle C_{j} \rangle} \cong (({\mathbb Z}^{2}_{2}*_{\mathbb Z})*_{\mathbb Z})*_{\mathbb Z}$ (a basic virtual Schottky group of type (T7)).

\subsubsection{The structural description of $K$}
Let $W \in \widetilde{\mathcal F}$ be the common boundary structure loop of the two regions $\widehat{R}_{i}$ and $\widehat{R}_{j}$.

If $K(W)=\{I\}$, then Klein-Maskit's combination theorem asserts that $\langle K^{*}_{i},K^{*}_{j}\rangle= K^{*}_{i} * K^{*}_{j}$.
If $K(W)=K^{*}_{i} \cap K^{*}_{j} \cong {\mathbb Z}_{2}$, then $\langle K^{*}_{i},K^{*}_{j}\rangle= K^{*}_{i} *_{K(W)} K^{*}_{j}$.

By doing this process at all pair of regions with common boundary structure loop, we obtain a subgroup $K^{*}<K$ which is a free product of groups as described in the theorem (the only type not used so far is (T2)). 

Next, if $W_{1} \in \widetilde{\mathcal F}$ is a boundary structure loop of $\widetilde{R}$, which has not been already considered, then (by the maximilaity choice of the regions $\widehat{R}_{j}$ and that they are non-$K$-equivalent) there exists some element $A \in K$ and a boundary structural loop $W_{2} \in \widetilde{\mathcal F}$ (not necessarily different from $W_{1}$) such that $A(W_{1})=W_{2}$. (If $W_{1} \neq W_{2}$, then $A$ is loxodromic and, if $W_{1}=W_{2}$, then $A$ has order two). The group $\langle K^{*}, A\rangle$ is (again by Klein-Maskit's combination theorem) the free product $K^{*} *\langle A \rangle$ (when $A$ is loxodromic, we are getting groups of type (T2)). 
We proceed with all the structural boundary loops in the similar way to obtain a subgroup $\widehat{K}<K$ which is a free product, in the sense of Klein-Maskit's combination theorem) ofbasic virtual Schottky groups.  As $\overline{\widetilde{R}}$ projects onto all $S/H=\Omega/K$, we observe that $\widehat{K}=K$ and we are done.


\end{document}